\documentclass[12pt]{article}
 \usepackage[margin=1in]{geometry} 
\usepackage{amsmath,amsthm,amssymb,amsfonts}
 \usepackage{amssymb}

\usepackage[utf8]{inputenc}
\usepackage[pagewise]{lineno}
\newcounter{cnstcnt}

\usepackage{hyperref}
 \usepackage{csquotes}
\usepackage[utf8]{inputenc}
\usepackage[english]{babel}
\usepackage{xcolor}
\usepackage{biblatex}
\providecommand{\keywords}[1]
{
  \small	
  \textbf{\textit{Keywords---}} #1
}
\title{ Boundedness in a chemotaxis system with weakly singular sensitivity in dimension two with arbitrary sub-quadratic degradation sources }
\author{ Minh Le \\
 Department of Mathematics\\
 Michigan State University\\
  Michigan, MI, 48823 \\
  \texttt{leminh2@msu.edu} }
\date{\today}

\begin{document}
\maketitle
\begin{abstract}  
We study the global existence and boundedness of solutions to a chemotaxis system with weakly singular sensitivity and sub-logistic sources in a two dimensional domain. X. Zhao (Nonlinearity; 2023; 36; 3909-3938 ) showed that the logistic degradation, $-\mu u^2$, can prevent blow-up under the largeness assumption on $\mu$. In this paper, we improve the result by replacing the quadratic degradation by sub-logistic one, $-\frac{\mu u^2}{\ln^\beta(u+e)}$ with $\beta \in (0,1)$, and removing the largeness assumption on $\mu$. 

\end{abstract}

\numberwithin{equation}{section}
\newtheorem{theorem}{Theorem}[section]
\newtheorem{lemma}[theorem]{Lemma}
\newtheorem{remark}{Remark}[section]
\newtheorem{Prop}{Proposition}[section]
\newtheorem{Def}{Definition}[section]
\newtheorem{Corollary}{Corollary}[theorem]
\allowdisplaybreaks

\section{Introduction}
We consider the following chemotaxis system in a smooth bounded domain $\Omega \subset \mathbb{R}^2$:
 \begin{equation} \label{1}
      \begin{cases}
          u_t = \Delta u - \nabla \cdot (u \chi(v) \nabla v) +f(u),  \qquad &\text{in } \Omega \times (0,T_{\rm max}) \\
          0 =  \Delta v - v +u, \qquad &\text{in } \Omega \times (0,T_{\rm max})  \\
          u(x,0) =u_0(x ),  \qquad &\text{in } \Omega \times (0,T_{\rm max})  \\
           \frac{\partial u}{\partial \nu } = \frac{\partial v}{ \partial \nu} =  0 \qquad &\text{on } \partial \Omega \times (0,T_{\rm max}) 
      \end{cases}
  \end{equation}
  where $r \in \mathbb{R}$, $\mu>0$, $\chi \in C^1((0,\infty))$, $f \in C([0, \infty))$, $T_{\rm max} \in (0, \infty]$ is the maximal existence time, and 
  \begin{align} \label{initial}
      u_0 \in C^0 (\bar{\Omega}) \text{ is nonnegative with } \int_\Omega u_0 >0.
   \end{align}
The system of equation \eqref{1} characterizes the motion of cells influenced by chemical cues. In this context, the functions $u(x,t)$ and $v(x,t)$ denote the cell density and chemical concentration respectively, at position $x$ and time $t$. When $\chi(v) \equiv \chi$ and $f \equiv 0$, this system is commonly referred to as the Keller-Segel system \cite{Keller}. Interestingly, the two-dimensional version of this system showcases a noteworthy characteristic referred to as the critical mass phenomenon. It stipulates that if the initial mass is below a certain threshold, solutions remain globally bounded \cite{Dolbeault, Dolbeault1, NSY}. Conversely, if the initial mass surpasses this threshold, solutions undergo finite-time blow-up \cite{Nagai1, Nagai2, Nagai4, Nagai3}.\\
To avert blow-up phenomena, logistic sources are introduced, where $f(u) = ru -\mu u^2$. This inclusion ensures the global boundedness of solutions across spatial dimensions \cite{Tello+Winkler, Winkler-logistic}. Notably, within two-dimensional domains, employing sub-logistic sources such as $f(u)= ru -\frac{\mu u^2}{\ln^\beta(u+e)}$ with $\beta \in (0,1]$ effectively prevents blow-up scenarios across various chemotaxis models \cite{Tian4, Tian2, Minh1, Minh2, Minh3, Dai+Xiang}. \\
In cases of singular sensitivity where $\chi(v) = \frac{\chi}{v^\alpha}$ with $\alpha >0$ and $f \equiv 0$, solutions are established to be globally bounded provided $\chi>0$ is sufficiently large in relation to $n$ and $\alpha$ \cite{Tobias, Fujie+Senba, KMT}. For the two-dimensional scenario with logistic sources, $f(u)= ru -\mu u^2$, the global existence of solutions has been established in \cite{FWY} when $\alpha=1$ for any $\mu >0$, and furthermore , if $r$ exceeds a certain threshold, solutions remain globally bounded in time. In the case where $\alpha \in (0,1)$, \cite{Zhao2023} has demonstrated that solutions are globally bounded given sufficiently large $\mu$. However, our paper presents findings indicating that the largeness assumption of $\mu$ can be relaxed. Our main results read as follows:
\begin{theorem} \label{thm1}
    Assume that $\chi(v)= \frac{1}{v^\alpha}$ with $\alpha \in (0,1)$, and $f(u)= ru -\frac{\mu u^2}{\ln^\beta(u+e)}$ with $\mu>0 $, $r\in \mathbb{R}$, and $\beta \in (0,1)$, then the
    system \eqref{1} admits a global classical solution $(u,v)$ such that 
    \begin{align}
        u &\in C^0 \left ( \bar{\Omega} \times [0, \infty) \right ) \cap C^{2,1} \left ( \bar{\Omega} \times (0,\infty) \right ) \notag \\
        v & \in C^{2,0} \left ( \bar{\Omega} \times [0, \infty ) \right ),
    \end{align}
    and such that both $u$ and $v$ are positive in $\bar{\Omega} \times (0, \infty)$. Moreover, the solution remains bounded in the sense that 
    \begin{align}
        \sup_{t>0} \left \| u(\cdot,t)  \right \|_{L^\infty(\Omega)} <\infty.
    \end{align}
\end{theorem}
We adhere closely to the framework outlined in \cite{Zhao2023} for our proof. However, we strategically leverage the sub-quadratic degradation terms to derive an $L\ln L$ estimate for solutions. Subsequently, employing a modified Gagliardo–Nirenberg interpolation inequality, we obtain an $L^p$ bound for solutions. Through this approach, we circumvent the need for the largeness assumption regarding the parameter $\mu$.

\section{A priori estimates and proof of theorem \ref{thm1}}
The local existence of solutions to the system \eqref{1} in the classical sense can be established by slightly modifying the arguments in \cite{FWY}[Lemma 2.2].
\begin{lemma} \label{lc}
    The system \eqref{1} under the assumptions of Theorem \ref{thm1} admits a unique classical solution $(u,v) \in C^0 \left ( \bar{\Omega} \times [0, T_{\rm max}) \right ) \cap C^{2,1} \left ( \bar{\Omega} \times (0, T_{\rm max}) \right ) \times C^{2,0} \left ( \bar{\Omega} \times [0, T_{\rm max} ) \right )$ where $T_{\rm max} \in (0, \infty]$ such that both $u>0$ and $v>0$ in $\bar{\Omega} \times (0,T_{\rm max})$, and such that 
    \begin{align}
        \text{if } T_{\rm max} <\infty \text{ then either } \limsup_{t\to T_{\rm max}} \left \| u(\cdot,t) \right \|_{L^\infty(\Omega)} = \infty \text{ or } \liminf_{t \to T_{\rm max}} \inf_{x\in \Omega} v(x,t) =0.
    \end{align}
\end{lemma}
From now on, we assume that $(u,v)$ is a local classical solution of \eqref{1} as clarified in Lemma \ref{lc}.  We proceed by establishing an $L\ln L$ estimate for solutions when $\alpha<\frac{1}{2}$, all without relying on the largeness assumption concerning $\mu$:
\begin{lemma} \label{l1}
    Assume that $\alpha <\frac{1}{2}$ and $\beta <1$, then there exists a positive constant $C$ such that 
    \begin{equation}\label{l1-1}
       \int_\Omega u (\cdot,t) \ln (u(\cdot,t)+e) \leq C
    \end{equation}
    for all $t \in (0,T_{\rm max})$.
\end{lemma} 
\begin{remark}
Lemma \ref{l1} resembles \cite{Zhao2023}[Lemma 3.2], albeit without the necessity of a largeness condition on $\mu$. This is because the term $-\mu \int_\Omega u^2 \ln^{1-\beta}(u+e)$ can absorb the nonlinear term stemming from chemo-attractant.
\end{remark}
\begin{proof}
Making use of integration by parts and the first equation of \eqref{1}, we obtain
\begin{align}\label{l1.1}
    \frac{d}{dt} \int_\Omega u \ln (u+e) &= \int_\Omega \left ( \ln(u+e) +\frac{u}{u+e} \right ) u_t \notag \\
    &= \int_\Omega \left ( \ln(u+e) +\frac{u}{u+e} \right ) \left ( \Delta u -  \nabla \cdot (\frac{u}{v^\alpha} \nabla v) +ru - \frac{\mu u^2}{ \ln^\beta  (u+e) } \right )  \notag \\
    &= - \int_\Omega \left ( \frac{1}{u+e}+ \frac{e}{(u+e)^2} \right )|\nabla u|^2+ \int_\Omega \left ( \frac{1}{u+e}+ \frac{e}{(u+e)^2} \right ) \frac{u}{v^\alpha} \nabla u \cdot \nabla v \notag \\
    &+ \int_\Omega \left (  \ln(u+e) +\frac{u}{u+e}  \right ) \left (  ru -\frac{\mu u^2 }{\ln^\beta (u+e)}\right ) \notag \\
    &:= I_1 +I_2+I_3.
\end{align}
    It is clear that
    \begin{align}\label{l1.2}
        I_1:= - \int_\Omega \left ( \frac{1}{u+e}+ \frac{e}{(u+e)^2} \right )|\nabla u|^2 \leq -\int_\Omega \frac{|\nabla u|^2 }{u+e}.
    \end{align}
    Using Holder's inequality leads to 
    \begin{align} \label{l1.3}
        I_2 &\leq c_1 \int_\Omega v^{-\alpha } |\nabla u| |\nabla v| \notag \\
        &\leq \frac{1}{2} \int_\Omega \frac{|\nabla u|^2}{u+e} + c_2 \int_\Omega \frac{(u+e)|\nabla v|^2}{v^{2\alpha}}
    \end{align}
    where $c_2 = \frac{c_1^2}{2}$. From \cite{Zhao2023}[Lemma 2.3], we have that $\int_\Omega \frac{|\nabla v|^2}{v^2} \leq |\Omega|$. This, together with Holder's inequality and Young's inequality implies that
    \begin{align}\label{l1.4}
        c_2\int_\Omega \frac{(u+e)|\nabla v|^2}{v^{2\alpha}} &\leq c_2\left ( \int_\Omega (u+e)^2 \right )^{\frac{1}{2}} \left ( \int_\Omega \frac{|\nabla v|^2}{v^2} \right )^\alpha \left ( \int_\Omega |\nabla v|^{\frac{4(1-\alpha)}{1-2\alpha}} \right )^{\frac{1-2\alpha}{2}} \notag \\
        &\leq c_3 \int_\Omega (u+e)^2 +  \left ( \int_\Omega |\nabla v|^{\frac{4(1-\alpha)}{1-2\alpha}} \right )^{1-2\alpha},
    \end{align}
    where $c_3 = c_2 |\Omega|^\alpha$. By the same reasoning as in \cite{Zhao2023}[Lemma 3.2], there exists $c_4>0$ such that 
    \begin{align}\label{l1.5}
         \left ( \int_\Omega |\nabla v|^{\frac{4(1-\alpha)}{1-2\alpha}} \right )^{1-2\alpha} \leq c_4 \int_\Omega u^2.
    \end{align}
    From \eqref{l1.3}, \eqref{l1.4} and \eqref{l1.5}, it follows that 
    \begin{align} \label{l1.6}
        I_2 &\leq \frac{1}{2} \int_\Omega \frac{|\nabla u|^2}{u+e} + c_3 \int_\Omega (u+e)^2 + c_4 \int_\Omega u^2 \notag \\
        &\leq \frac{1}{2} \int_\Omega \frac{|\nabla u|^2}{u+e}+ \frac{\mu}{2} \int_\Omega u^2 \ln^{1-\beta } (u+e) +c_5,
    \end{align}
    where $c_5= C(\mu)>0$ and the last inequality comes from the fact that for any  $\epsilon>0$, there exist $c(\epsilon)>0$ such that 
\begin{align} \label{M}
    u^{a_1}\ln^{b_1}(u+e) \leq \epsilon u^{a_2}\ln^{b_2}(u+e) +c(\epsilon),
\end{align}
where $a_1, a_2, b_1,b_2$ are nonnegative numbers such that $a_1 <a_2$, or $a_1=a_2$ and $b_1<b_2$. Applying \eqref{M} and choosing $\epsilon$ sufficiently small yields
\begin{align} \label{l1.7}
    I_3 +\int_\Omega u \ln(u+e) \leq \frac{ \mu}{2} \int_\Omega u^2 \ln^{1-\beta }(u+e) +c_6,
\end{align}
where $c_6=C(\mu)>0$. Collecting \eqref{l1.2}, \eqref{l1.6}, and \eqref{l1.7} implies that 
\begin{align} 
    \frac{d}{dt} \int_\Omega u \ln(u+e) +\int_\Omega u\ln (u+e) \leq c_7,
\end{align}  \label{l1.8}
where $c_7=c_5+c_6$. Finally, the inequality \eqref{l1-1} is proven by applying Gronwall's inequality to \eqref{l1.8}.    
\end{proof}
The next lemma allows us to obtain an $L\ln L $ estimate for solutions for any $\alpha \in (0,1)$ and $\mu>0$.
\begin{lemma} \label{l2}
    Assume that $\alpha \in \left [ \frac{1}{2}, 1 \right ) $ and $\beta <1$, then there exists a positive constant $C$ such that 
    \begin{equation}\label{l2-1}
       \int_\Omega u (\cdot,t) \ln (u(\cdot,t)+e) \leq C
    \end{equation}
    for all $t \in (0,T_{\rm max})$.
\end{lemma}
\begin{remark}

The proof of Lemma \ref{l2} closely mirrors \cite{Zhao2023}[Lemma 3.6], with the only distinction being our utilization of $-\mu \int_\Omega u^2 \ln^{1-\beta}(u+e)$ instead of solely $-\mu \int_\Omega u^2$, as employed in \cite{Zhao2023}.
\end{remark}

\begin{proof}
    Suppose that $\alpha \in \left ( \frac{2^k-1}{2^k}, \frac{2^{k+1}-1}{2^{k+1}} \right )$ with $k\geq 1$, we set 
    \begin{align} \label{l2.1}
        \alpha_j = 2^j \alpha +1 -2^j
    \end{align}
    then we have from \cite{Zhao2023}[Lemma 3.3] that
    \begin{align} \label{l2.2}
        \alpha_0 &= \alpha, \quad \alpha_j \in \left ( \frac{1}{2},1 \right ), j=1,2,..., k-1, \notag \\
        \alpha_k &\in \left (0,\frac{1}{2} \right ), \quad \alpha_{k+1} \in (-1,0).
    \end{align}
    Following the proof of Lemma 3.3 and Lemma 3.4 in \cite{Zhao2023} with some modifications, one can verify that there exist $c_1,c_2,c_3>0$ such that
    \begin{align} \label{l2.3}
        \int_\Omega (u+e) v^{-2 \alpha} |\nabla v|^2 &\leq \frac{1}{2} \int_\Omega \frac{|\nabla u|^2}{u+e} +c_1 \int_\Omega u^2 +c_2 \notag \\
        &\leq \frac{1}{2} \int_\Omega \frac{|\nabla u|^2}{u+e} + \frac{\mu}{2} \int_\Omega u^2 \ln^{1-\beta }(u+e)+c_3,
    \end{align}
    where the last inequality comes from \eqref{M}. Following the same argument in the proof of Lemma \ref{l1} but replacing the estimates \eqref{l1.4} and \eqref{l1.5} by \eqref{l2.3} proves \eqref{l2-1}. In case $\alpha = \frac{2^k-1}{2^k}$, we repeat the same argument in Remark 3 in \cite{Zhao2023} to deduce \eqref{l2-1}, which completes the proof.
\end{proof}
\begin{lemma} \label{l3'}
    Let $\mu>0$, $r\in \mathbb{R}$ and $\alpha \in \left ( 0,1 \right )$  then 
    \begin{align} \label{l3'.1}
        \sup_{t\in (0, T_{\rm max})} \int_\Omega |\nabla v(\cdot,t)|^2 <\infty.
    \end{align}
\end{lemma}
\begin{proof}
     Thanks to Lemma \ref{l1} and Lemma \ref{l2}, we have that 
    \begin{align}
         \sup_{t \in (0, T_{\rm max})}\int_\Omega u(\cdot,t) \ln u(\cdot,t) <\infty,
    \end{align}
    for any $\mu>0$. Now, we just repeat the arguments in the proof \cite{Zhao2023}[Lemma 3.7] to prove \eqref{l3'.1}. 
\end{proof}
We can now establish that the local classical solution $(u,v)$ is global for any $\alpha \in (0,1)$.  However, since the proof closely resembles that of \cite{Zhao2023}[Lemma 3.8], we will refrain from presenting it here to avoid redundancy.
\begin{lemma}
    Let $\mu>0$, $\alpha \in (0,1)$, $\beta \in (0,1)$, and $r \in \mathbb{R}$ and suppose that $\eqref{initial}$ holds. Then the problem \eqref{1} possesses a uniquely global classical solution.
\end{lemma}
The following lemma serves as tool to obtain the $L^p$ bounds with $p\geq 1$ for solutions of \eqref{1}. It is a direct consequence of \cite{Winkler_preprint}[Corollary 1.2], however for the convenience, we provide the detail proof here.
\begin{lemma} \label{C52.LGN} If $\Omega \subset \mathbb{R}^2$ is a bounded domain with smooth boundary, then for each $m>0$ and $\gamma\geq 0$ there exists $C=C(m,\gamma)>0$ with the property that whenever $\phi \in C^1 (\bar{\Omega})$ is positive in $\bar{\Omega}$
    \begin{align}\label{C52.LGN.1}
        \int_\Omega \phi^ {m+1} \ln^{\gamma}(\phi+e) \leq C \left ( \int_\Omega \phi \ln^{\gamma}(\phi+e) \right ) \left ( \int_\Omega |\nabla \phi^{\frac{m}{2}}|^2 \right  ) +C \left ( \int_\Omega  \phi \right )^m \left ( \int_\Omega \phi \ln^{\gamma}(\phi+e) \right ).
    \end{align}
\end{lemma}
\begin{proof}
By applying Sobolev's inequality when $n=2$, there exists a postive constant $c_1$ such that
\begin{align} \label{C52.LGN-1}
    \int_\Omega \phi^{m+1} \ln^\gamma (\phi+e) &\leq c_1 \left ( \int_\Omega  \left |\nabla \left ( \phi^{\frac{m+1}{2} }\ln^{\frac{\gamma}{2}} (\phi+e) \right ) \right | \right )^2+c_1 \left ( \int_\Omega \phi \ln^{\frac{\gamma}{m+1}}(\phi+e) \right )^{m+1} 
\end{align}
By using elementary inequalities, one can verify that 
\begin{align*}
  \left |\nabla \left ( \phi^{\frac{m+1}{2} }\ln^{\frac{\gamma}{2}} (\phi+e) \right ) \right | \leq c_2 \phi^{\frac{1}{2}} \ln^{\frac{\gamma}{2}}(\phi+e) |\nabla \phi^{\frac{m}{2}}|,
\end{align*}
where $c_2=C(m, \gamma)>0$. This, together with Holder's inequality leads to 
\begin{align} \label{C52.LGN-3}
    c_1 \left ( \int_\Omega  \left |\nabla \left ( \phi^{\frac{m+1}{2} }\ln^{\frac{\gamma}{2}} (\phi+e) \right ) \right | \right )^2 \leq c_3 \int_\Omega |\nabla \phi^{\frac{m}{2}}|^2 \cdot \int_\Omega \phi \ln^{\gamma}(\phi+e),
\end{align}
where $c_3=c_1c_2$. By Holder's inequality, we deduce that
\begin{align}\label{C52.LGN-4}
    c_1 \left ( \int_\Omega \phi \ln^{\frac{\gamma}{m+1}}(\phi+e) \right )^{m+1}  \leq  c_1 \left ( \int_\Omega  \phi \right )^m \left ( \int_\Omega \phi \ln^{\gamma}(\phi+e) \right ).
\end{align}
Collecting \eqref{C52.LGN-1}, \eqref{C52.LGN-3} and \eqref{C52.LGN-4} implies \eqref{C52.LGN.1}, which finishes the proof.
\end{proof}
As a consequence, we have the following lemma:
\begin{lemma}\label{C52.ILGN} Assume that $\Omega \subset \mathbb{R}^2$ is a bounded domain with smooth boundary  and $p>0$, $\gamma>\xi \geq 0$. For each $\epsilon>0$, there exists $C=C(\epsilon,\xi,\gamma)>0$ such that the following inequality holds
    \begin{align}\label{C52.ILGN.1}
        \int_\Omega \phi^ {m+1} \ln^{\xi}(\phi+e) \leq \epsilon \left ( \int_\Omega \phi \ln^{\gamma}(\phi+e) \right ) \left ( \int_\Omega |\nabla \phi^{\frac{m}{2}}|^2 \right  ) +\epsilon \left ( \int_\Omega  \phi \right )^m \left ( \int_\Omega \phi \ln^{\gamma}(\phi+e) \right ) +C,
    \end{align}
     for any positive function $\psi \in C^1 (\bar{\Omega})$.
\end{lemma}

\begin{proof}
    Since $\gamma >\xi \geq 0$, one can verify that for any $\delta>0$, there exists $c_1=c(\delta, \xi, \gamma)>0$ such that for any $a \geq 0$ we have
    \begin{align}
        a^{m+1} \ln^\xi (a+e) \leq \delta a^{m+1} \ln^ \gamma (a+e)+c_1.
    \end{align}
    This entails that
    \begin{align}
        \int_\Omega \phi^ {m+1} \ln^{\xi}(\phi+e) \leq \delta \int_\Omega \phi^ {m+1} \ln^{\gamma}(\phi+e)+c_1|\Omega|.
    \end{align}
    Now for any fixed $\epsilon$, we choose $\delta = \frac{\epsilon}{C}$ where $C$ as in Lemma \ref{C52.LGN}, and apply \eqref{C52.LGN.1} to have the desire inequality \eqref{C52.ILGN.1}.
\end{proof}
The following step is to establish $L^p $ bounds for solutions with $p>1$.
\begin{lemma} \label{l3}
For any $p>1$ there exists $C=C(p)>0$ such that
    \begin{equation}
        \sup_{t \in (0,T_{\rm max})} \int_\Omega u^p (\cdot,t) \leq C
    \end{equation}
\end{lemma}
\begin{proof}
    A direct calculation shows
    \begin{align}\label{l3.1}
        \frac{1}{p}\frac{d}{dt} \int_\Omega u^p &= \int_\Omega u^{p-1} \left ( \Delta u - \nabla \cdot (\frac{u}{v^\alpha} \nabla v) +ru -\frac{\mu u^2}{\ln^\beta (u+e)} \right ) \notag  \\
        &= - \frac{4(p-1)}{p^2} \int_\Omega |\nabla u^{\frac{p}{2}}|^2 + \frac{p-1}{p} \int_\Omega \frac{u^{\frac{p}{2}}}{v^\alpha} \nabla u^{\frac{p}{2}} \cdot \nabla v +r \int_\Omega u^p - \mu \int_\Omega \frac{u^{p+1}}{\ln^\beta (u+e)} .
    \end{align}
    By Young's inequality, we obtain
    \begin{align}\label{l3.2}
         \frac{p-1}{p} \int_\Omega \frac{u^{\frac{p}{2}}}{v^\alpha} \nabla u^{\frac{p}{2}} \cdot \nabla v \leq \frac{p-1}{p^2} \int_\Omega |\nabla u^{\frac{p}{2}}|^2+ c_1 \int_\Omega \frac{u^p}{v^{2\alpha}} |\nabla v|^2,
    \end{align}
    where $c_1=C(p)>0$.
    Following carefully the arguments in the proof of Lemma 4.2 and Lemma 4.3 in \cite{Zhao2023} with some necessary modifications yields that 
    \begin{align}\label{l3.3}
         c_1 \int_\Omega \frac{u^p}{v^{2\alpha}} |\nabla v|^2 \leq \frac{1}{2} \int_\Omega |\nabla u^{\frac{p}{2}}|^2 + c_2 \int_\Omega u^{p+1}+c_3,
    \end{align}
    where $c_2, c_3>0$. 
Since $\sup_{t\in (0,T_{\rm max) }}\int_\Omega u \ln (u+e) < \infty$, we can apply Lemma \ref{C52.ILGN} with $p=m$, $\gamma =1$, $\xi=0$, and
    \begin{equation*}
        \epsilon = \frac{p-1}{p^2 c_2\sup_{t\in (0,T_{\rm max) }}\int_\Omega u \ln (u+e)}
    \end{equation*}
    to deduce that
    \begin{align}\label{l3.4}
        c_2 \int_\Omega u^{p+1} &\leq \epsilon c_2 \int_\Omega |\nabla u^{\frac{p}{2}}|^2 \int_\Omega u\ln (u+e) + \epsilon c_2 \left ( \int_\Omega u \right )^m \int_\Omega u \ln (u+e) +c_4 \notag \\
        &\leq \frac{p-1}{p^2} \int_\Omega |\nabla u^{\frac{p}{2}}|^2  +c_5,
    \end{align}
    where $c_4,c_5>0$ depending on $\epsilon$. Now, one can verify by using elementary inequality that
    \begin{align}\label{l3.5}
        \left ( r+ \frac{1}{p} \right ) \int_\Omega u^p \leq \frac{\mu}{2} \int_\Omega \frac{u^{p+1}}{\ln^\beta (u+e)} +c_6,
    \end{align}
    where $c_6=C(\mu)>0$. Fron \eqref{l3.1} to \eqref{l3.5}, we deduce that 
    \begin{align}\label{l3.6}
        \frac{1}{p} \frac{d}{dt} \int_\Omega u^p + \frac{1}{p} \int_\Omega u^p \leq c_7
    \end{align}
    where $c_7= c_3+c_5+c_6$. Finally, the proof is completed by applying Gronwall's inequality to \eqref{l3.6}.
\end{proof}
We are now ready to prove the main result.
\begin{proof}[Proof of Theorem \ref{thm1}]
The proof follows precisely the same reasoning as in \cite{Zhao2023}, with the exception of Lemmas \ref{l3'} and \ref{l3}, which do not necessitate the largeness assumption on $\mu$ as in Lemmas 4.2 and 4.4 in \cite{Zhao2023}.
\end{proof}
\section*{Acknowledgments}
The author acknowledges support from the Mathematics Graduate Research Award Fellowship at Michigan State University.
\printbibliography

\end{document}